\documentclass[11pt]{article}
\usepackage{setspace}
\usepackage{amsmath}
\usepackage[margin=1in]{geometry}
\usepackage{amssymb}
\usepackage{amsmath}
\usepackage{microtype}
\usepackage{graphicx}
\usepackage{hyperref}
\usepackage{amsthm}
\usepackage{tikz}
\newtheorem{lm}{Lemma}
\newtheorem{thm}[lm]{Theorem}
\newtheorem{defi}[lm]{Definition}
\newtheorem{conj}[lm]{Conjecture}

\newtheorem{cor}[lm]{Corollary}
\newtheorem{eg}[lm]{Example}

\usepackage[font=small,labelfont=bf]{caption}

\title{A Conjectural Brouwer Inequality for Higher-Dimensional Laplacian Spectra}
\author{Rediet Abebe \\ Cornell University \& Harvard University \\ \url{red@cs.cornell.edu}}

\begin{document}

\maketitle
\begin{abstract}
We present a generalization of Brouwer's conjectural family of inequalities -- a popular family of inequalities in spectral graph theory bounding the partial sum of the Laplacian eigenvalues of graphs -- for the case of abstract simplicial complexes of any dimension. We prove that this family of inequalities holds for shifted simplicial complexes, which generalize threshold graphs, and give tighter bounds (linear in the dimension of the complexes) for simplicial trees. We prove that the conjecture holds for the the first, second, and last partial sums for all simplicial complexes, generalizing many known proofs for graphs to the case of simplicial complexes. We also show that the conjecture holds for the $t^{\text{th}}$ partial sum for all simplicial complexes with dimension at least $t$ and matching number greater than $t$. Returning to the special case of graphs, we expand on a known proof to show that the Brouwer's conjecture holds with equality for the $t^{\text{th}}$ partial sum where $t$ is the maximum clique size of the graph minus one (or, equivalently, the number of cone vertices). Along the way, we develop machinery that may give further insights into related long-standing conjectures.
\end{abstract}

\section{Introduction}\label{sec:intro}

Let $G$ be a finite, simple, undirected graph on $n$ vertices $\{v_1, v_2, \ldots, v_n\}$ and with edge set $e(G)$. The degree of a vertex, denoted by $\deg(v_i)$, is the number of edges incident to the vertex. The \emph{degree sequence} of the graph is defined as $d(G) := \left( { \deg(v_1), \deg(v_2), \ldots, \deg(v_n) } \right)$. The \emph{conjugate partition} of the graph is the sequence $d^{T}(G) := \left( { d^{T}_1, d^{T}_2, \ldots } \right)$, where $d^{T}_i$ is the number of vertices with degree at least $i$. Throughout this paper, we assume that $d^T(G)$ is of length $n$. Therefore, $d^T(G)$ will contain at least one $d_i^T$ with value $0$. $G$ can be represented using its Laplacian matrix:

\begin{defi}
The \emph{Laplacian matrix} of $G$, denoted by $L(G)$, is an $n \times n$ matrix where each entry $\ell_{i, j}$ is

\[
\ell_{i, j} := \begin{cases} \deg(v_i) & \text{if }\, i = j \\-1 & \text{if }\, i \neq j, \text{ and } v_i \text{ is adjacent to } v_j, \\0 & \text{otherwise.}\end{cases}
\]

\end{defi}

The rows and columns of the Laplacian matrix are indexed by the vertices of $G$.
Laplacian matrices are known to be positive-semidefinite, and therefore have non-negative, real numbers as eigenvalues \cite{BH}. We denote the eigenvalues of $L(G)$, known as the \emph{Laplacian spectrum}, by
\[ \lambda(G)= \left( \lambda_1(G) \geq \lambda_2 (G) \geq \ldots \geq \lambda_n (G) \right).\]
It is also known that the multiplicity eigenvalue $0$ is equal to the number of connected components of $G$. In particular, there is at least one $0$ eigenvalue, $\lambda_n=0$ .

In each of the above notations, we drop the $G$ when there is no ambiguity and simply use $e, d, d^T, L,$ and $\lambda$ to refer to the edge set, degree sequence, conjugate partition, Laplacian matrix, and Laplacian spectrum of $G$, respectively.

In spectral graph theory, there is an interest in bounding these eigenvalues \cite{BH,Merris,S,Z}. These include well-known results such as that $\lambda_1 \leq n$, $\sum_i \lambda_i = 2e$, and that the largest eigenvalue $\lambda_1$ is at most twice the maximum degree of a vertex in the graph. There are also several long-standing conjectures bounding these eigenvalues. A popular one related to the sum of the Laplacian eigenvalues is Brouwer's conjecture \cite{BH}.

\begin{conj}[Brouwer's Conjecture]
\label{bc}
Let $G$ be a graph with $e$ edges and Laplacian spectrum $\lambda = \left( \lambda_1 \geq \lambda_2 \geq \ldots \lambda_{n-1} \geq \lambda_n \right)$. Then, $\forall t \in \{1, 2, \ldots, n\} =: [n]$,

\[ \sum_{i=1}^{t} \lambda_i \leq e + \binom{t+1}{2}.\]
\end{conj}

Despite being posed in 2008, this conjectural family of inequalities remains open except for a few special classes of graphs, such as trees, regular graphs, co-graphs, and graphs of a given matching number or maximum degree. It is a straightforward consequence of the above observations on the largest Laplacian eigenvalue that this conjecture holds for $t = 1$. Likewise, Brouwer's conjectural family of inequalities holds for $t = n-1$ and $n$ due to the observation on the sum of the Laplacian eigenvalues of graphs. Separately, the conjecture has also been shown to hold for $t = 2$ in \cite{H}.

In this work, we first expand on a known proof of Brouwer's conjecture for threshold graphs to show that the conjectural family of inequalities holds with equality when $t$ is the maximum clique size of the graph minus one. This is discussed in Section \ref{sec:thresh}. We then revert our attention to the more general case of abstract simplicial complexes of any dimension and present a conjectural family of inequalities which generalize Brouwer's conjecture. We define abstract simplicial complexes, which are the objects of interest, in Section \ref{sec:simpcomp}. We then introduce the new conjectural family of inequalities for simplicial complexes and discuss them in Section \ref{sec:gbc}. The conjecture is also stated below:

\begin{conj}
\label{red}
Given an (abstract) simplicial complex $S$ of dimension $k-1$ and Laplacian eigenvalues $\lambda(S) = \left( \lambda_1 \geq \lambda_2 \ldots \right)$, for all $t \in \mathbb{N}$,
\[ \sum_{i=1}^{t} \lambda_i \leq (k - 1) f_{k-1}+ \binom{t+k -1}{k}\]
where $f_{k-1}$ is the number of $k-1$ dimensional faces in $S$.
\end{conj}

Brouwer's conjecture is a special case of Conjecture \ref{red} when $k = 2$, where $f_1$ corresponds to the number of edges in the graph. In Section \ref{sec:gbc}, we show that this conjecture holds for various classes of simplicial complexes such as shifted simplicial complexes, which generalize threshold graphs. We also give a tighter bound (linear in the number of dimensions of the simplicial complex) for simplicial trees. We show that the conjectural family of inequalities holds for $t = 1, 2,$ as well as the last partial sums. These generalize many known proofs for graphs to the case of simplicial complexes of any dimension. We further show that the conjecture holds for the $t^{\text{th}}$ partial sum for all simplicial complexes with dimension at least $t$ and matching number greater than $t$. En route, we develop machinery that may give further insights into related conjectures on the Laplacian eigenvalues of simplicial complexes. These conjectures are stated and discussed in Section \ref{sec:related}.

\section{Related Work}\label{sec:related}

While the conjecture has been open since it was posed in 2008, steady progress has been made to resolve Brouwer's conjecture for several classes of graphs. Known cases include trees, where tighter bounds have been given \cite{tree2,GZW,H}, regular graphs \cite{M}, unicycle, bicycle, and tricycle graphs \cite{DZ,WHL,YY}, co-graphs \cite{M}, graphs of a given clique size, vertex covering, diameter, matching number, and maximum degree \cite{CLF,DMG,DZ,G2,RV}.
In addition, the Brouwer's conjecture is also shown to hold for $t = \{1, 2, n-1, n\}$, where $n$ is the number of vertices, as well as for connected graphs when $t$ is sufficiently large \cite{C2,H,YY}.

Brouwer's conjecture is intimately related to several conjectures in spectral graph theory. A closely related conjecture is that posed by Grone and Merris in 1994, which remained unresolved until the 2010 paper by Bai \cite{Bai,GM}. This theorem states that the Laplacian spectrum of $G$ is majorized by the conjugate partition of $G$. Or, equivalently,
\[
\sum^{t}_{i=1} \lambda_i \leq \sum^{t}_{i=1} d^T_i.
\]
Throughout this paper, we will refer to this result as Bai's theorem.

A noteworthy observation is that Bai's theorem needs significantly more information than Brouwer's conjecture; namely, the former requires the conjugate partition of the graph while the latter only uses the number of edges. However, the bound given by Brouwer's conjecture is known to be sharper than that given by Bai's theorem for a given graph $G$ and $t$ if and only if $G$ is a split graph: i.e., a graph where the vertices can be partitioned into a clique and an independent set \cite{M}.

Other conjectures on the partial sum of the Laplacian eigenvalues of graphs include Zhou's conjecture, which looks at the sum of the powers Laplacian eigenvalues of graphs \cite{Z}. The partial sum of the Laplacian eigenvalues of graphs is connected with the Laplacian energy, as discussed in \cite{CLF,GZ,lapenergy,Z2}.

Related to simplicial complexes, the Grone-Merris conjecture was generalized for simplicial complexes by Duval and Reiner \cite{DR}, which is discussed in Section \ref{sec:gbc}. While the conjecture is resolved for graphs by Bai, the generalized conjecture, which we refer to as the Duval-Reiner conjecture, remains open for simplicial complexes. (This conjecture is discussed in \ref{sec:simpcomp} after introducing the necessary notations.) An exception is the special case of shifted simplicial complexes, which generalize threshold graphs \cite{DR}. The techniques developed in this work, and especially those related to matching in simplicial complexes, may yield further insights into the Duval-Reiner conjecture for further classes of simplicial complexes.

\section{Brouwer's Conjecture for Threshold Graphs}\label{sec:thresh}

Before stating the generalized conjecture and presenting results related to it, we present a result related to Brouwer's conjecture for threshold graphs. We first review two operations on graphs that will be useful for defining threshold graphs:

Given two graphs $G_1$ and $G_2$ with disjoint vertex sets, we can obtain a new graph which is the \emph{disjoint union} of these. We denote this graph by $G := G_1 \sqcup G_2$.\footnote{We use $\sqcup$ rather than $\cup$ to denote the disjoint union of graphs (and later simplicial complexes) that have disjoint vertex sets.} This graph has the disjoint union of the vertex sets of $G_1$ and $G_2$ as its vertex set, and likewise for its edge set. The Laplacian matrix of this new graph $G$ is the direct sum of the Laplacian matrices of $G_1$ and $G_2$, and therefore the Laplacian spectrum of $G$ is the disjoint union of the Laplacian spectra of $G_1$ and $G_2$.

The \emph{complement graph} of $G$, which we denote by $G^c$, is a graph on the same vertex set as $G$ where two vertices are adjacent in $G^c$ if and only if they are not adjacent in $G$. If the Laplacian spectrum of $G$ is $\left( \lambda_1 \geq \lambda_2 \geq \ldots \geq \lambda_n \right)$, then that of $G^c$ is $\left( 0 \leq n-\lambda_{n-1} \leq \ldots \leq n-\lambda_1\right)$, as shown in \cite{BH}. We relabel the Laplacian spectrum of $G^c$ by $\left( \lambda_1^c \leq \ldots \leq \lambda_{n-1}^c \leq \lambda_n^c=0\right)$.

These operations allow us to define certain classes of graphs:

\begin{defi}
A \emph{complement-reducible graph} (or co-graph) is a graph that is inductively built using the following rules:
\begin{itemize}
\item A single vertex is a co-graph,
\item If $G$ is a co-graph, then so is $G^c$, and
\item If $G_1$ and $G_2$ are co-graphs on disjoint vertex sets, then so is $G_1 \sqcup G_2$.
\end{itemize}
\end{defi}

We note the following lemma, whose proofs can be found in \cite{M} or reproduced using the above statements about how we obtain the Laplacian spectrum of $G_1 \sqcup G_2$ and also of $G^c$.

\begin{lm}
\label{comp}
A graph $G$ satisfies Brouwer's conjecture if and only if $G^c$ also satisfies Brouwer's conjecture.
\end{lm}

\begin{lm}
\label{disjoint}
If two disjoint graphs $G_1$ and $G_2$ satisfy Brouwer's conjecture, then so does $G = G_1 \sqcup G_2$.
\end{lm}

\begin{lm}
Co-graphs satisfy Brouwer's conjectural family of inequalities.
\end{lm}

\begin{proof}
This is a consequence of the fact that a single vertex satisfies Brouwer's conjectural family of inequalities and an inductive application of Lemma \ref{comp} and Lemma \ref{disjoint}.
\end{proof}

Threshold graphs are a special subclass of co-graphs. They are defined in the same way as co-graphs, except we require $G_1$ to be an isolated vertex. Equivalently, threshold graphs are inductively constructed one vertex at a time, where each vertex is either an isolated vertex or a cone vertex (i.e., a vertex connected to all existing vertices at the time of addition). Two special examples of threshold graphs are the complete graph on $n$ vertices, $K_n$, which is obtained by adding a cone vertex at each step and the star graph on $n$ vertices, $S_n$, where we add $n-1$ isolated vertices followed by a final cone vertex. A known result for this class of graphs is due to Merris \cite{M}, which states:

\begin{thm}
Given a threshold graph $S$ with Laplacian eigenvalues $\lambda(S) = \left( \lambda_1 \geq \lambda_2 \geq \ldots \lambda_n \right)$ and conjugate partition $d^T(S) = \left(d_1^T, d_2^T, \ldots, d_n^T \right)$,
$$\sum^{t}_{i=1} \lambda_i = \sum^{t}_{i=1} d^T_i,$$
for all $t \in [n]$.
\end{thm}

We show an example of a threshold graph construction and demonstrate Merris' theorem for threshold graphs.
\begin{eg}

We consider the following construction of a threshold graph built through a sequence of vertex additions: $\left(cone, isolated, isolated, cone \right)$, shown in Figure 1.

\vspace{2mm}
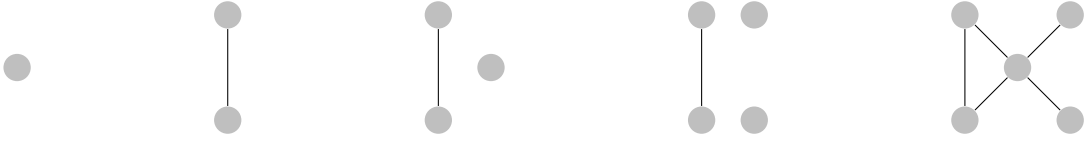
\begin{figure}[!ht]
\centering
\begin{tikzpicture}
{
\begin{tikzpicture}
[scale=0.7,auto=left,every node/.style={circle,fill=gray!50}]
\node (n1) at (0,1) {};

\node (n2) at (4, 0) {};
\node (n3) at (4, 2) {};

\node (n4) at (8, 0) {};
\node (n5) at (8, 2) {};
\node (n6) at (9, 1) {};

\node (n7) at (13, 0) {};
\node (n8) at (13, 2) {};
\node (n9) at (14, 0) {};
\node (n10) at (14, 2) {};

\node (n11) at (18, 0) {};
\node (n12) at (18, 2) {};
\node (n13) at (20, 0) {};
\node (n14) at (20, 2) {};
\node (n15) at (19, 1) {};

\draw[thin] (n2) edge[bend left=0] (n3);
\draw[thin] (n7) edge[bend left=0] (n8);
\draw[thin] (n4) edge[bend left=0] (n5);
\draw[thin] (n11) edge[bend left=0] (n12);
\draw[thin] (n15) edge[bend left=0] (n11);
\draw[thin] (n15) edge[bend left=0] (n12);f
\draw[thin] (n15) edge[bend left=0] (n13);
\draw[thin] (n15) edge[bend left=0] (n14);
\end{tikzpicture}

};
\end{tikzpicture}
\caption{A threshold graph constructed by adding a cone, isolated, isolated, and cone node, in order. }
\end{figure}

This last graph on $5$ vertices has conjugate partition $d^T = \left( 5, 3, 1, 1,0 \right) $. We also find that the Laplacian spectrum is $\lambda = \left( 5, 3, 1, 1, 0\right)$, consistent with Merris' theorem for threshold graphs. Note, this graph also satisfies Brouwer's conjecture for all $t \in [5]$ and satisfies it with equality for $t = 2$.

\end{eg}

Since threshold graphs are a special class of co-graphs, Brouwer's conjectural family of inequalities holds for threshold graphs. Here, we extend this proof to show that Brouwer's conjecture holds with equality on threshold graphs when $t$ is the number of cone vertices in the above construction.

\begin{lm}
Given a threshold graph $S$, Brouwer's conjecture holds with equality for $t = c$, where $c$ is the number of cone vertices in the threshold graph construction (or equivalently the maximum clique size of $S$ minus 1).
\end{lm}

\begin{proof}
We show this by induction on the construction of $S$. Suppose $S$ is constructed one vertex at a time resulting in a sequence of threshold graphs $S_1, S_2, \ldots, S_n$, where $S_1$ is an isolated vertex and $S_n = S$. Consider the first instance of a cone vertex, $\ell$. The corresponding $S_\ell$ is a star graph. Star graphs on $m$ vertices are known to have Laplacian spectrum $\left(m, 1, 1, \ldots, 1,0\right)$, where eigenvalue $1$ has multiplicity $m-2$. Therefore, Brouwer's conjectural family of inequalities holds with equality for this class of graphs for $t = 1$.

For the inductive step, suppose that $S_j$ satisfies Brouwer's conjecture with equality for $t = c_j$, where $c_j$ is the number of cone vertices in $S_j$. That is,
\[ \sum^{c_j}_{i=1} d_i^T (S_j)= e (S_j) + \binom{c_j+1}{2}.\]
If $T_{j+1}$ is the disjoint union of $S_j$ and an isolated vertex, then $S_{j+1}$ continues to satisfy the conjecture with equality for the same $t$ by Lemma \ref{disjoint}. If $S_{j+1}$ is the result of a cone of $S_j$, then $e(S_{j+1}) = e(S_j) + j$ and,
$$\sum^{c_j}_{i=1} d_i^T(T_j) + c_j + j + 1 = \sum^{c_j+1}_{i=1} d_i^T (S_{j+1}).$$
Therefore,
\begin{align*}
\sum^{c_j+1}_{i=1} d_i^T (S_{j+1})& = e(S_j) + \binom{c_j+1}{2} + c_j + j + 1\\
& = e(S_{j+1}) + \binom{c_j+1}{2} + c_j + 1\\
& = e(S_{j+1}) + \binom{(c_j+1) + 1}{2}
\end{align*}
giving us the desired equality.
\end{proof}

\section{Laplacian Spectrum for Abstract Simplicial Complexes}\label{sec:simpcomp}

We now turn our attention to the more general setting of simplicial complexes, which will allow us to present the new conjecture and our main results.

\begin{defi}
An \emph{(abstract) simplicial complex} $S$ on a vertex set $\{v_1, v_2, \ldots, v_n\}$ is a collection of subsets of this vertex set, which we call the \emph{faces} or \emph{simplices}, that are closed under inclusion. That is, given $F$ which is a subset of this vertex set, if $F' \subseteq F$, then $F'$ is also in $S$.
\end{defi}

Given a face $F$, we denote its cardinality by $|F|$. Its dimension is $|F| - 1$. For instance, the face $\{v_1, v_2, v_3\}$ has dimension $2$. The dimension of a simplicial complex is the maximum dimension of a face in $S$. Note that we will interchangeably refer to $S$ as a $k$-family or a $k-1$-dimensional complex. A discussion on why it suffices to look at $k$-families can be found in \cite{DR}. Graphs are a special case when $k = 2$. Throughout this paper, we assume that all simplicial complexes are of dimension $k-1$ (or are $k$-families), unless explicitly stated otherwise.

\begin{defi}
The $f$-vector of a simplicial complex $S$ is the sequence,
$$f(S) = (f_{-1}(S), f_{0}(S), f_1(S), \ldots ), $$
where $f_i(S)$ is the number $i$-dimensional faces.
\end{defi}

We do not need to assume that all singletons of $\{v_1, v_2, \ldots, v_n\}$ are included in $S$. Therefore, the number of vertices $f_0(S)$ need not be $n$. We assume that $\{\emptyset\}$, which is the empty set and no other faces, is always included in $S$. Therefore, $f_{-1}(S) = 1$. If $S$ is a $k$-family, then $f_j(S) = 0$ for all values $j \geq k$.

The maximal faces under inclusion are called \emph{facets}. All $k$-families are assumed to be \emph{pure}, i.e., all facets have the same dimension. Therefore, $f_{k-1}(S)$ is the number of facets in $S$. It is straightforward to generalize our results by dropping the purity assumption.

The \emph{degree} of a vertex $v_i$, denoted by $\deg(v_i)$, is the number of facets in which the vertex is included. We adapt the definitions of degree sequence $d(S)$ and conjugate partition $d^T(S)$ from the graph case. Note that $d(S)$ has dimension $n$ and $d^T$ has dimension $\binom{n - 1}{k - 1} +1$. As in the case for graphs, we drop the $S$ from our notations when there is no ambiguity.

For some results, it will also be useful to consider subfamilies of $k$-families.

\begin{defi}
A \emph{$k$-subfamily} $H$ of $S$ is a subfamily such that every face in $H$ is also a face in $S$. Note that we assume $H$ is also closed under inclusion and it has the same dimension as $S$.
\end{defi}

\begin{eg}
A $k-1$-dimensional \emph{complete simplicial complex} (or complete $k$-family) on $n \geq k$ vertices has, as its facets, the complete $k$-family on the vertex set $\{v_1, v_2, \ldots, v_n\}$. That is, it has $\binom{n}{k}$ facets, each of which correspond to subsets of $\{v_1, \ldots, v_k\}$ of size $k$. This $k$-family has degree sequence $d = \left( \binom{n-1}{k-1}, \binom{n-1}{k-1}, \ldots, \binom{n-1}{k-1} \right)$ and conjugate partition $d^T = \left( n, n \ldots, n, 0 \right)$.

A $k-1$-dimensional \emph{star simplicial complex} (or star $k$-family) on the vertex set $\{v_1, v_2, \ldots, v_n\}$ is the $k$-family with facets,
$$\{ \{v_1, v_2, \ldots, v_{k-1}, v_{k}\}, \{v_1, v_2, \ldots, v_{k-1}, v_{k+1}\}, \ldots, \{v_1, v_2, \ldots, v_{k-1}, v_{n}\} \}.$$
It has $n-k+1$ facets. Exactly $k$ of the vertices have degree $n-k+1$ and the remaining have degree $1$. Figure 2 depicts a star $2$-family on $5$ vertices. It has degree sequence $d(S) = (3, 3, 1, 1, 1)$ and conjugate partition $d^T(S) = (5, 2, 2, 0, 0)$.

\begin{figure*}[!ht]
\centering
\includegraphics[width=5cm]{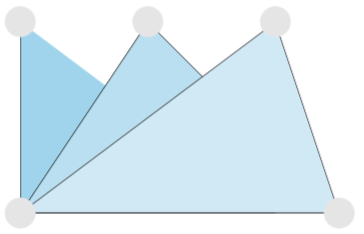}
\caption{an example of a star $k$-family: a star $2$-family on $5$ vertices.}
\end{figure*}

\end{eg}

To define the Laplacian spectrum of simplicial complexes, we recall from \cite{hatcher} that, given a simplicial complex $S$, we have chain groups $C_i(S)$ and simplicial maps between these chain group:
$$0 \rightarrow \ldots \xrightarrow{\partial_3}C_2(S) \xrightarrow{\partial_2} C_1(S) \xrightarrow{\partial_1} C_0(S) \xrightarrow{\partial_0} 0.$$
The $C_i$ are vector spaces with the basis being the $i$-dimensional faces of the simplicial complex. The $\partial$ are called \emph{(simplicial) boundary maps}. The boundary map $\partial_{k-1}$ goes from the vector space whose basis is the $k-1$-dimensional face to the vector space whose basis is the $k-2$-dimensional faces defined as,
$$\partial_i[v_0, \ldots, v_{k-1}] = \sum_{j = 0}(-1)^j [v_0, \ldots, v_{-j}, \ldots v_{k-2}].$$
Note that the basis elements $[v_0, \ldots, v_i]$ are regarded as having $[v_{w(0)}, \ldots, v_{w(i)}] = \text{sgn} (w) [v_0, \ldots, v_i]$ when indexing the basis element for a simplex $\{v_0, \ldots, v_i\}$ with its vertices in different orders. Furthermore, $\partial_i \partial_{i +1} = 0$, and so we can define the homology group $H_n(S) = \text{ker}(\partial_i) / \text{im} \partial_{i + 1}$.

These boundary maps can be written down as matrices. For instance, $\partial_2$ can be written down as a matrix whose columns are indexed by the $2$-dimensional faces and whose rows are indexed by the $1$-dimensional faces.

\begin{defi}
Given a $k-1$-dimensional simplicial complex $S$ with boundary maps defined as above, the \emph{Laplacian matrix} of $S$ is,
$$L(S) = \partial_{k - 1} \partial_{k - 1}^T.$$
\end{defi}

The boundary map $\partial_1$ is known as the oriented incidence matrix for graphs. And $\partial_1 \partial_1^T$ gives us an alternate way of defining the Laplacian matrix for graphs.

Given a simplicial complex $S$, its Laplacian spectrum is therefore the eigenvalues of the Laplacian matrix as defined above. We denote this by,
$$\lambda(S) = \left( \lambda_1 \geq \lambda_2 \geq \ldots \geq \lambda_{\binom{n}{k-1 }} = 0\right).$$
We are interested in the partial sum of the Laplacian spectrum of simplical complexes.

As in the discussion for graphs, we first present some results showing how the Laplacian spectra of simplicial complexes change when we consider some standard operations on simplicial complexes. Namely, we consider taking their disjoint union, complement, and simplicial join. Understanding how the Laplacian spectra change under each of these three basic operations will each play a key role in proving some of the main results in this work.

The first operation we consider is taking the \emph{disjoint union} of simplicial complexes or $k$-families. Before presenting this result, we note the following theorem about real, symmetric matrices.

\begin{thm}[Fan's Theorem]
Given real, symmetric matrices $A$ and $B$ of size $m$,
$$\sum_{i = 1}^t \lambda_i(A + B) \leq \sum_{i = 1}^t \lambda_i(A) +\sum_{i = 1}^t \lambda_i(B), \forall t \in \{1, 2, \ldots, m\},$$
where $\lambda_i (A)$ corresponds to the $i^{\text{th}}$ largest Laplacian eigenvalue of $A$ (likewise for $B$ and $A+ B$).
\end{thm}

The following result is a direct consequence of Fan's theorem.

\begin{cor}\label{cor:fan}
Suppose $S_1, S_2, \ldots, S_r$ are disjoint $k$-families on the same vertex set, i.e., they share no $k$-subsets. Let $\lambda_i^j$ be the $i^{\text{th}}$ largest eigenvalue of $S_j$ and $S$ be the $k$-family $\cup_{j = 1}^r S_j.$ Then,
$$\sum_{i =1}^t \lambda_i(S) \leq \sum_{j = 1}^r \sum_{i = 1}^t \lambda_i^j.$$
\end{cor}

This corollary holds since $L(S ) = L(S_1) + L(S_2) + \ldots + L(S_r)$, where we index the rows and columns of each of these Laplacian matrices the same way. We can then apply Fan's theorem to $L(S)$.

For this above result to hold, the $k$-families in consideration have to be, in a sense, disconnected. We define a generalized notion of connectedness in this setting, which is called \emph{ridge-connectedness}. To define this notion, we first construct ridge graphs from $k$-families.

\begin{defi}
A \emph{ridge graph} of a $k$-family is a graph that has a vertex corresponding to each facet of the $k$-family and an edge between two vertices if and only if the corresponding facets intersect in co-dimension one.
\end{defi}

Note that ridge graphs do not uniquely identify $k$-families. To see this, take the following two $3$-families, $S = \{ \{1,2, 3\}, \{1, 2, 4\}, \{1, 3, 4\}, \{2, 3, 4\}\}$ and $T = \{\{1,2, 3\}, \{1, 2, 4\}, \{1, 2, 5\}, \{1, 2, 6\}\}$. While $S \neq T$, these $k$-families both have $K_4$, the complete graph on four vertices, as their ridge graph.

\begin{defi}
A $k$-family $S$ is ridge-connected if its ridge graph is a connected graph. 
\end{defi}

Henceforth, we may drop the ``ridge" and simply say that a $k$-family is connected or disconnected. 

Putting all of this together, if $S$ is not ridge-connected, then you can decompose its set $R$ of co-dimension one faces (or ridges) into a disjoint union $S_1 \sqcup S_2 \sqcup \ldots \sqcup S_c$, where $c$ is the number of connected components of the ridge graph of $S$. Note that the vertex set of $S$ is the union of the vertex sets of the $S_i$ and $F$ is a face in $S$ if and only if it is face in one of the $S_i$. By Corollary \ref{cor:fan}, $L(S) = L(S_1) + L(S_2) + \ldots + L(S_c)$. The Laplacian spectra of $S$ is the disjoint union of the Laplacian spectra of the $S_i$.

Another operation that builds a new $k$-family from an existing one is to take the \emph{complement}. We define the complement of a $k$-family $S$, denoted by $S^c$, as follows:
$$S^c := \binom{[n]}{k} \backslash S = \{F \subseteq [1, 2, \ldots, n]: |F| = k, F \notin S\}.$$
$S^c$ is also a $k$-family and $(S^c)^c = S$. If $\lambda(S) = ( \lambda_1, \lambda_2, \ldots)$, then the Laplacian spectrum of $S^c$ is given by
$$\lambda_i(S^c) = n - \lambda_{\binom{n -1}{k - 1} + 1 - i}$$
as shown in \cite{DR}.

\begin{defi}
Suppose $S_1$ and $S_2$ are families on non-intersecting vertex sets $[n]$ and $[m]$, where $S_1$ and $S_2$ may potentially be of different dimensions. Their \emph{simplicial join}, denoted by $S = S_1 \star S_2$ is a $k$-family on vertex set $[m + n]$, where for each face $\{v_1, v_2, \ldots, v_i\} = F_1 \in S_1$ and each face $\{w_1, w_2, \ldots, w_j\} = F_2 \in S_2$, $S$ has a face $F = \{v_1, v_2, \ldots, v_i, w_1, w_2, \ldots, w_j\}$.
\end{defi}

Note that the dimension of $S$ is the sum of the dimensions of $S_1$ and $S_2$. When $S_2$ is an isolated vertex, the simplicial join of $S_1$ and $S_2$ is called the \emph{simplicial cone} of $S_1$. A result from \cite{DR} that will be useful in this work is how Laplacian spectra change under simplicial coning. Namely, given a $k$-family $S$ with Laplacian spectrum $\{\lambda_1, \lambda_2, \ldots, \lambda_n=0 \}$, the Laplacian spectrum of its cone $\hat{S}$ is $\{\lambda_1 +1, \lambda_2 +1, \ldots , \lambda_n + 1 = 1, \lambda_{n+1} = 0\}$.

\begin{eg}
Suppose $S_1$ is the $k$-family on $n-k$ isolated vertices and $S_2$ is the complete $k$-family on $[k]$. Then $S_1 \star S_2$ is the star $k$-family on $[n]$ nodes.

Suppose $S$ is the complete $k$-family on $[n]$. The cone of $S$ is the complete $k+1$-family on vertex set $[n + 1]$.
\end{eg}

\section{A Generalized Conjecture for Abstract Simplicial Complexes}\label{sec:gbc}

We can now introduce our conjectural Brouwer inequality for higher-dimensional Laplacian spectra. We note that this conjecture generalizes Brouwer's conjecture for graphs.

\vspace{2mm}

\noindent \textbf{Conjecture 3.} Given a simplicial complex (or $k$-family) $S$ and Laplacian spectrum $\lambda(S) = \left( \lambda_1 \geq \lambda_2 \geq \ldots \right)$, for all $t \in \mathbb{N}$,
\[ \sum_{i=1}^{t} \lambda_i \leq (k - 1) f_{k-1}+ \binom{t+k -1}{k}\]
where $f_{k-1}$ is the number of $k-1$-dimensional faces in $S$.

\vspace{2mm}

Without loss of generality, we assume that all $k$-families considered here are pure and connected. The former assumption is by construction of Laplacian matrices from the previous section. The latter is by application of Fan's theorem, which allows us to consider connected $k$-families separately.

We first show that this conjecture holds for the first and last partial sums.

\begin{lm}
The generalized family of inequalities stated in Conjecture \ref{red} holds for $t = 1, \binom{n}{k - 2}-1$, and $\binom{n}{k - 2}$.
\end{lm}

\begin{proof}

For $t = 1$, we want to show
$$\lambda_1 \leq (k-1)f_{k-1} + 1.$$
Since $f_{k - 1}\geq n - k + 1$ by the connectedness assumption and $\lambda_1 \leq n$, as shown in \cite{BH}, we have:
\begin{align*}
\lambda_1 &\leq (k-1)f_{k-1} + 1\\
n & \leq (k - 1)(n - k + 1) + 1\\
n & \leq nk - k^2 + 2k - n
\end{align*}
This above inequality gives us $k \leq n$, which holds since we assume that $S$ is connected. 

To prove the higher partial sums, we make use of the fact from \cite{BH} that $\sum_i \lambda_i = k f_{k-1}(S)$ for all $k$-families where the smallest eigenvalue is $0$. Since the smallest eigenvalue is $0$, we only need to prove this for $t = \binom{n}{k - 2} -1$.
\begin{align*}
\sum_{i = 1}^{\binom{n}{k - 2} -1} \lambda_i & \leq (k - 1) f_{k -1} + \binom{t + k - 1}{k}\\
k f_{k - 1} & \leq (k - 1) f_{k -1} + \binom{t + k - 1}{k}\\
f_{k - 1} & \leq \binom{t + k - 1}{k}
\end{align*}
To show that this last inequality holds, we first note that we can assume $k > 2$. Namely, the case for $k = 2$ corresponds to the case of graphs, for which Brouwer's conjecture has already been shown to hold for the last partial sums. Therefore, $\binom{n}{k - 2} -1 \geq n - 1$. Therefore, the right hand side gives us $\binom{n - 2 + k}{k}$ whereas the left hand side is bounded above by $\binom{n}{k}$.

\end{proof}

In the last subsection, we will show that this conjectural family of inequalities also holds for $ t = 2$ as well as for higher partial sums under certain dimension and matching number constraints. Before doing so, we show that it holds for all $t$ for some special classes of $k$-families.

\subsection{Shifted Simplicial Complexes}

We define a special class of $k$-families, which generalize threshold graphs.

\begin{defi}
$S$ is a \emph{shifted simplicial complex} (or shifted $k$-family) if there exists a labeling of its vertices by $[n]$ such that for any face $\{v_1, v_2, \ldots, v_d\}$, replacing any of the $v_i$ by a label $v_j$ such that $j < i$ results in a face that is also in $S$.
\end{defi}

We can also define this class using order ideals.

\begin{defi}
A non-empty subset $P'$ of a partially ordered set $P$ is called an \emph{order ideal} if for all $x \in P'$, $y \leq x$ implies that $y \in P'$. Moreover, for every $x, y \in P'$, there is some element $z \in P'$ such that $x \leq z$ and $y \leq z.$

\end{defi}

Consider a poset $P$ on $n$ integers such that a string $(x_1 < x_2 < \ldots < x_k)$ is said to be less than another string $(y_1 < y_2 < \ldots y_k)$ if $x_i \leq y_i$ for all $i$. Shifted $k$-families on $n$ vertices of degree $k$ are precisely the order ideals of $P$. This ordering is sometimes known as the \emph{Gale ordering}.

Shifted $k$-families are a combinatorial class that have interesting algebraic and topological properties. It is known that shifted $k$-families satisfy the Duval-Reiner conjecture and, in fact, a stronger result is shown to hold \cite{DR}:

\begin{thm}[Duval-Reiner]
Let $S$ be a shifted $k$-family with conjugate partition $d^T = (d_1^T, d_2^T, \ldots)$ and Laplacian spectrum $\lambda = \{ \lambda_i, \forall i \in [n]\}$. Then,
$$\sum_{i = 1}^t \lambda_i = \sum_{i = 1}^t d_i^T.$$
\end{thm}

In this section, we show that shifted $k$-families satisfy Conjecture \ref{red} as well.

\begin{thm}\label{thm:shifted}
Shifted $k$-families satisfy the family of inequalities given in Conjecture \ref{red}.
\end{thm}

\begin{proof}
We will do this via an induction on the cardinality of the order ideal in the Gale ordering that indexes the shifted $k$-family by removing Gale maximal $k$-subsets one at a time.

Suppose that Conjecture 3 is not satisfied for some $t$ and a $k$-family $\tilde{S}$. We choose a minimal counterexample $S \subseteq \tilde{S}$ for which this conjecture does not hold. That is, if we remove a maximal facet $F$ of $S$, then the resulting $k$-family $S \backslash F$ satisfies Conjecture 3. Note, $F$ is assumed to be maximal in the ideal within the Gale order as defined above since, if it is not, then the new $k$-family $S \backslash F$ will not be shifted. We have two cases:

\begin{itemize}
\item[Case 1.] S has at least $t + k$ vertices.

There exists a maximal face $F = \{v_1, v_2, \ldots, v_k\}$ where $v_1 \geq v_2 \ldots \geq v_k$ such that $v_k$ is at least $t + k$.
To see this, choose any face such that $v_k \geq t + k$. If it is maximal, then we are done. If it is not, then we can increase the indices by going up the partially ordered set until we reach a maximal element. Relabel this face $ F = \{v_1, v_2, \ldots, v_k\}$. We still note $v_k \geq t + k$ since going up the partially ordered set can only increase the indices of the vertices.

Since $S$ is shifted, it follows that it contains $F' = \{v_1, v_2, \ldots, v_{k -1}, v_{k'}\}$ for all $v_{k'} < v_k$ such that $F'$ is still a $k$-family. Thus, each of $\{v_1, v_2, \ldots, v_{k-1}\}$ has degree at least $t+ 1$.

We delete $F$ from $k$-family to get $S \backslash F$, reducing the degree of $\{v_1, v_2, \ldots, v_{k - 1}\}$ by one each. Thus, when we delete $F$, the left hand side of Conjecture 3 goes down by less than $k - 1$ while the right hand side goes down by exactly $k - 1$ by the Duval-Reiner theorem. Therefore, if $S \backslash F$ satisfies Conjecture 3, then so does $S$.

\item[Case 2.] $S$ has fewer than $t + k$ vertices. Then, we have at most $\binom{t+k - 1}{k}$ faces. Thus,
$$k f_{k - 1} \leq \binom{t + k - 1}{k} + (k - 1) f_{k - 1}.$$

But, $\sum_{i = 1}^t \lambda_i \leq k f_{k - 1}$, so the desired inequality holds.

\end{itemize}
\end{proof}

\subsection{Simplicial Trees}

Returning to the case of Brouwer's conjecture for graphs, it is known that the following stronger result holds for trees, as shown in \cite{H}:
$$ \sum_{i = 1}^t \lambda_i \leq e + 2t - 1.$$

We generalize this above family of inequalities and show that this result holds for simplicial trees. Before presenting this result, we first define simplicial trees as introduced by Faridi, \cite{F}.

\begin{defi}
A facet $F$ of a $k$-family is called a \emph{leaf} if either $F$ is the only facet of the $k$-family $T$ or $ F \cap T \backslash F \subseteq G$ for some facet $G \in T \backslash F$.
\end{defi}

\begin{defi}
A connected $k$-family $T$ is a \emph{tree} if every $k$-subfamily of $T$ has a leaf.
\end{defi}

Simplicial forests are defined by dropping the connectedness assumption above.

\begin{thm}
Let $T$ be a simplicial tree on $n$ vertices with Laplacian spectrum $( \lambda_1 \geq \lambda_2 \geq \ldots)$. Then,
$$\sum_{i = 1}^t \lambda_i \leq (k - 1) f_{k - 1} + k t - k + 1.$$
\end{thm}

\begin{proof}

We prove this by induction on $n$.
We first assert that if $T$ is a star $k$-family, which is known to have Laplacian spectrum $( n, k - 1, k - 1, \ldots, k -1, 0 )$, then it satisfies the desired inequality above since
$$ n + (t - 1)(k - 1) \leq (k - 1)(n - k + 1) + kt - k + 1$$
holds for a given $t$ whenever $n \geq k$.

Suppose that $T$ is not a star $k$-family. Then, it must have at least three facets. Given such a $T$, there exists a facet whose removal results in a simplicial forest with two components $T_1$ and $T_2$ which have at least one facet each. To see this, first note that there must exist a facet $F$ that is not a leaf since $T$ is not a star $k$-family. Suppose that this facet corresponds to the vertex $v_f$ in the ridge graph of $T$, which we denote by $G_T$. Now, suppose that $v_f$ has neighboring vertices $\{v_1, v_2, \ldots, v_m\} \in V(G_T)$. Now, consider removing $F$ from $T$, and correspondingly $v_f$ from $G_T$. Then, there exists $v_i, v_j \in \{v_1, v_2, \ldots, v_m\}$ such that no $v_i - v_j$ path in $G_T$ that does not contain $v_f$. If there was such a path $C = \{v_i, \ldots, v_j\}$, the $k$-subfamily corresponding to $C$, which we denote b $S_C$, would not contain a leaf, contradicting the assumption that $T$ is a simplicial tree.

Now, given the $t$ largest eigenvalues of $T_1 \sqcup T_2$, say $t_1$ of them come from $T_1$ and $t_2$ of them come from $T_2$. We have one of two cases. Note, for each case, $K_k$ denotes a $k$-family with just one facet $\{v_1, v_2, \ldots, v_k\}$.

\begin{itemize}
\item[Case 1.] One of the $t_i$, say $t_2$, is 0. Then, by Corollary 16, and the inductive hypothesis,
\begin{align*}
\sum_{i = 1}^t \lambda_i (T) &\leq \sum_{i = 1}^t \lambda_i(T_1 \cup T_2 \cup K_k)\\
& \leq \sum_{i = 1}^t \lambda_{i} (T_1) + \sum_{i = 1}^t \lambda_i (K_k)\\
& \leq (k - 1) f_{k - 1}(T_1) + k t_1 - k + 1 + k \\
& \leq (k - 1) f_{k - 1}(T_1) + k t_1 +1
\end{align*}
\noindent The last line implies the desired inequality since:
\begin{align*}
(k - 1) f_{k - 1}(T_1) + k t_1 +1 & \leq (k-1) f_{k - 1} + kt - k + 1\\
(k - 1) f_{k - 1}(T_1) + k t_1 & \leq (k-1) f_{k - 1}(T_1) + (k-1)f_{k-1}(T_2) + (k - 1) + kt - k \\
k t_1 & \leq (k-1)f_{k-1}(T_2) + kt - 1
\end{align*}
\noindent This last inequality holds since $f_{k-1}(T_2) \geq 1$ and $k \geq 2$ by assumption.

\item[Case 2.] $t_1, t_2 \neq 0$. Then, again by Corollary 16 and the inductive hypothesis,
\begin{align*}
\sum_{i = 1}^t \lambda_i (T) &\leq \sum_{i = 1}^t \lambda_i(T_1 \cup T_2 \cup K_k)\\
& \leq \sum_{i = 1}^t \lambda_i(T_1) + \sum_{i = 1}^t \lambda_i(T_2) + \sum_{i = 1}^t \lambda_i(K_k)\\
& \leq (k - 1) f_{k - 1}(T_1) + (k - 1) f_{k - 1}(T_2) + k t_1 + k t_2 - k + 2\\
& \leq (k - 1) f_{k - 1}(T) + k t - k + 1
\end{align*}
Note, the last line uses the observation that $f_{k - 1}(T_1) + f_{k - 1}(T_2) + 1 = f_{k - 1}(T)$ and $t_1 + t_2 = t$.
\end{itemize}
\end{proof}

\subsection{Higher Partial Sums}\label{sec:hps}

Next, we present two main results for higher partial sums of Conjecture \ref{red}. We have already noted that the conjecture holds for the first and last partial sums. Here, we note that the conjecture holds for all $t$ for $k$-families such that $k > t$ and whose matching number is greater than $t$. We also show that the conjecture holds for all $k$-families for $ t = 2$, which states:
\[ \lambda_1(S) + \lambda_2(S) \leq (k - 1)f_{k - 1} (S) + k + 1.\]

At the heart of these results is a structural argument that heavily relies on matching in $k$-families. Recall that a matching in a graph is a set of edges such that no two edges share a vertex. (i.e., the edges are independent.) The matching number of a graph is the size of a matching that contains the largest possible number of edges. We generalize this notion of matching numbers for $k$-families of any dimension below.

\begin{defi}
Given a $k$-family $S$ with a ridge graph $G_S$, a \emph{matching} is a set of facets in $S$ corresponding to an independent set in $G_S$. The \emph{matching number} of a $k$-family, denoted by $M_S$, is the size of the maximal independent set of its ridge graph.
\end{defi}

Note that this definition generalizes the notion of matching in graphs.

This notion of matching, along with the following observation on forbidden $k$-families, plays a key role in proving higher partial sums for Conjecture \ref{red}.

\begin{lm}\label{lm:forbidden}
If $k$-family $S$ is a minimum-cardinality counter example to the $t^{\text{th}}$ partial sum inequality in Conjecture 3, then every $k$-subfamily $H$ in $S$ must satisfy for $t > 1$:
$$\sum_{i = 1}^t \lambda_i(H) >(k - 1) f_{k - 1}(H).$$
\end{lm}

\begin{proof}
We prove this by contradiction. Let $S$ be such a counter-example with the minimum number facets. i.e.,
$$\sum_{i=1}^t \lambda_i (S) > (k-1)f_{k-1} (S) + k + 1. $$
Let $H$ be a $k$-subfamily of $S$. Then, $S \backslash H$ forms a $k$-family. By assumption, if $H$ satisfies $\sum_{i=1}^t \lambda_i (H) \leq (k-1)f_{k-1} (H)$, using the consequence of Fan's theorem, we have
\begin{align*}
\sum_{i=1}^t \lambda_i (H ) + \sum_{i=1}^t \lambda_i (S \backslash H) &\geq \sum_{i = 1}^t \lambda_i (S)\\
\sum_{i=1}^t \lambda_i (H ) + \sum_{i=1}^t \lambda_i (S \backslash H) & > (k-1)f_{k-1} (S) + \binom{t +k -1}{k}\\
(k-1) f_{k-1} (H) + \sum_{i=1}^t \lambda_i ( S \backslash H) &> (k-1)f_{k-1} (S) + \binom{t +k -1}{k}\\
\sum_{i=1}^t \lambda_i ( S \backslash H) & > (k-1)f_{k-1} (S \backslash H) + \binom{t +k -1}{k}
\end{align*}
Note, the second line is implied by $\sum_{i=1}^t \lambda_i (S) > (k-1)f_{k-1} (S) + k + 1.$

The last line implies that $S \backslash H$ is also a counter-example to the conjecture, contradicting minimality of $S$.
\end{proof}

\begin{thm}[Higher Partial Sums]\label{thm:hps}
The family of inequalities in Conjecture \ref{red} holds for all $t < k$ for $k$-families with matching number greater than $t$.
\end{thm}

\begin{proof}
We prove that there exists a $k$-subfamily $H$ such that $\sum_{i = 1}^t \lambda_i(H) \leq (k - 1) f_{k - 1}(H)$, which, coupled with Lemma \ref{lm:forbidden}, proves the desired result.

If a $k$-family $S$ has matching number $M_S$, there exists a $k$-subfamily $H = K_k \sqcup K_k \sqcup \ldots \sqcup K_k$, where $K_k$ is a $k$-family on $k$ nodes. Recall that this $H$ corresponds to an independent set in the ridge graph of $S$, denoted by $G_S$. We note,
\begin{align*}
\sum_{i = 1}^t \lambda_i(H) &\leq (k-1) f_{k - 1}(H)\\
tk & \leq (k -1) M_S
\end{align*}
This last line holds since $M_S > t$ and $t < k$, by assumption.
\end{proof}

\subsection{Second Partial Sum}\label{sec:second}

We now turn our attention to the case of $t = 2$ for all $k$-families. Recall that we assume $k$-families considered in this section are connected. To see that we do not lose generality by making this assumption, suppose $S$ is a $k$-family such that $S = S_1 \sqcup S_2$. Then, either the top two eigenvalues of $S$ both come from one of the $S_i$, say $S_1$, or they are the largest eigenvalues of $S_1$ and $S_2$. In the case of the former, we only have to prove the conjecture for $S_1$. In the case of the latter, we can make use of the fact that Conjecture \ref{red} holds for $t = 1$. That is:
\begin{align*}
\lambda_1(S) + \lambda_2(S) &= \lambda_1(S_1) + \lambda_1(S_2)\\
& \leq (k-1) f_{k - 1}(S_1) + (k - 1) f_{k - 1}(S_2) + 2\\
& = (k - 1)f_{k - 1}(S) + 2.
\end{align*}

Recall that since Brouwer's conjecture has already been shown to hold for graphs, we can assume that $k > 2$. Furthermore, in light of Theorem \ref{thm:hps}, we only need to focus on the case where $S$ has matching numbers 1 or 2.

\begin{lm}
A $k$-family $S$ has matching number 1 (i.e., has a complete ridge graph) if and only if $S$ is the simplicial join of $S_1 \star S_2$ of a $k_1$-family $S_1$ and a $k_2$-family $S_2$ (where $k = k_1 + k_2$) such that:
\begin{itemize}
\item $S_1$ has only one $k_1$-set and
\item $S_2$ is either a $1$-family, so a set of disjoint vertices, or a $k_2$-family for $k_2 \geq 2$ consisting of all $k_2$-sets of some $k_2 + 1$-set.
\end{itemize}
\end{lm}

\begin{proof}
For the forward direction, label the facets of $S$ by $\{F_1, F_2, \ldots, F_s\}$ and let $S_1 = \cap_{i = 1}^s F_s$. This is a full $k_1$-simplex (one $k_1$-set) where $k_1 \leq k$. Note that this $k_1$ may be 0. Take the facets $\{F_1 \backslash V(S_1), F_2 \backslash V(S_1), \ldots, F_s \backslash V(S_1)\}$ and relabel them with $\overline{F}_i = F_i \backslash V(S_1)$ for all $i \in [s]$. These faces are $k_2$-families, where $k_2 = k - k_1$. Let $S_2$ be the $k$-family whose facets are these $\overline{F}_i$. Any two facets $F_j$ and $F_k$ intersect in co-dimension one in $S$. Therefore, $\overline{F}_j$ and $\overline{F}_k$ must also intersect in co-dimension one, which implies that $S_2$ has a complete ridge graph. But, all the facets of $S_2$ do not have a common intersection. i.e., $\cap_{i = 1}^s \overline{F}_i = \emptyset.$ Therefore, $S_2$ is a complete $k$-family.


The backward direction is more straightforward. Assume that $S = S_1 \star S_2$ such that $S_1$ and $S_2$ satisfy the conditions given above. Then, since $S_2$ is a complete $k$-family, it has a complete ridge graph. And, since $S_1$ is just a full simplex, the simplicial join of the two will continue to have a complete ridge graph. Therefore, $S$ has matching number $1$.
\end{proof}

This result implies the following lemma:

\begin{lm}[Matching Number 1]
Conjecture \ref{red} holds on the second partial sum for $k$-families $S$ with matching number 1.
\end{lm}

\begin{proof}
By the above lemma, such a $k$-family $S$ is a simplicial join $S = S_1 \star S_2$. To see that this gives us a shifted $k$-family, label the indices of the facet in $S_1$ by $\{1, \ldots, k_1\}$ and the vertices in $S_2$ by $\{k_1 + 1, k_1+2, \ldots k_2\}$. Shiftedness follows from the fact that $S_2$ is a complete $k$-family and $S_1$ is just one simplex. Therefore, by Theorem \ref{thm:shifted}, the desired result holds.
\end{proof}

We now turn our attention to the final case -- $k$-families with matching number two.

First, note that if the ridge graph of a $k$-family is a tree, then the $k$-family is a tree, following our definition.
A leaf in a ridge graph corresponds to a leaf in a $k$-family and all subgraphs of a tree are also trees (or forests, if they have more than one connected component). Since Conjecture \ref{red} has already been shown to hold for simplicial trees, we can consider $k$-families whose ridge graphs are not trees.

Let $\ell$ be the minimal cycle length of a ridge graph of a $k$-family $S$. If $\ell > 6$, then $S$ has a matching number greater than $2$. Therefore, we can narrow our attention to $3 \leq \ell \leq 5$. We will consider each of these cases independently. Before doing so, we first note several forbidden $k$-subfamilies by Lemma \ref{lm:forbidden}.

\begin{lm}
Suppose $S$ is a $k$-family, where $k \geq 3$, with matching number $2$. Let $S$ have $k$-subfamilies $K_{k+1}$ (the complete $k$-family on vertex set $[k+ 1]$), $S_{2, 2} = S_2 \sqcup S_2$ (where $S_2$ is the star $k$-family with two facets), or $S_{3} \sqcup K_k$, where $S_3$ is the star k-family with three facets. Then, $S$ is not a minimum-cardinality counterexample to the $t = 2$ inequality for Conjecture \ref{red}.
\end{lm}

\begin{proof}
We reason about each of these forbidden subfamilies one at a time:
\begin{itemize}
\item[Case 1:] Consider $K_{k+1}$. Then, $\lambda_1 = \lambda_2 = k + 1$ and $f_{k-1}(K_{k+1}) = k + 1$.
\item[Case 2:] Consider $S_{2, 2}$. Here, $\lambda_1 = \lambda_2 = k + 1$, and $S_2 \sqcup S_2$ has four facets.
\item[Case 3:] And last, consider $S_{3} \sqcup K_k$. Note, $\lambda_1(S_3 \sqcup K_k) = k + 2$ and $\lambda_2(S_3 \sqcup K_k) = k$, while this $k$-family has four facets.
\end{itemize}

\noindent Since $k \geq 3$, by applying Lemma \ref{lm:forbidden}, we note that if any of the above are $k$-subfamilies of $S$, then $S$ cannot be a minimum-cardinality counterexample to Conjecture \ref{red} for $t = 2$.
\end{proof}

\begin{lm}[Matching Number 2]
Let $S$ be a $k$-family with matching number 2. Then, it satisfies Conjecture \ref{red} for $t = 2$.
\end{lm}

\begin{proof}
We prove the cases where $\ell = 3, 4,$ and $5$ independently. The argument proceeds by finding forbidden $k$-subfamilies for each of these cases and applying Lemma \ref{lm:forbidden} to contradict minimality of $S$.

Let $S$ have a ridge graph with a minimum cycle of length $3$. Denote the $k$-subfamily corresponding to this cycle by $C_3$. This $k$-subfamily has matching number $1$, while $S$ has matching number 2. Therefore, there exists a facet $F$ in $S$ such that $F$ and the $k$-subfamily corresponding to $C_3$ do not have a co-dimension one intersection. We therefore have $S_2 \sqcup K_k$ as a $k$-subfamily, where $S_2$ is the star $k$-family on $k+1$ vertices and $K_k$ is the complete $k$-family on $k$ vertices. The two largest eigenvalues of $S_2 \sqcup K_k$ are $k+1$ and $k$, where the former is the largest eigenvalue of $S_2$ and the latter is that of $K_k$. This is a forbidden $k$-subfamily for $k > 3$ since it violates Lemma \ref{lm:forbidden}. To note the case for $k = 3$, we consider the $k$-subfamily $C_3 \sqcup K_k$. This $C_3$ can either be $S_3$, the star $k$-family on $k +2$ vertices, or the following $k$-family:

\begin{figure}[!ht]
\centering
\begin{tikzpicture}[scale = 1.3]
\node[fill=gray] (n1) at (0, 0) {};
\node[fill=red] (n2) at (-1.2, -1) {};
\node[fill=red] (n3) at (1.2, -1) {};
\node[fill=red] (n4) at (0, 1) {};

\fill[fill=blue!20] (n4.center)--(n2.center)--(n3.center);
\path[draw] (n1)--(n2);
\path[draw] (n1)--(n3);
\path[draw] (n1)--(n4);
\path[draw] (n2)--(n3);
\path[draw] (n4)--(n3);
\path[draw] (n4)--(n2);
\node[circle,fill=gray!20] (n1) at (0, 0) {};
\node[circle,fill=gray!20] (n2) at (-1.2, -1) {};
\node[circle,fill=gray!20] (n3) at (1.2, -1) {};
\node[circle,fill=gray!20] (n4) at (0, 1) {};
\end{tikzpicture}
\end{figure}

\noindent If $C_3$ is the star $k$-family on $k+2$ vertices, then $C_3 \sqcup K_k$ contains $S_3 \sqcup K_k$ as a $k$-subfamily, which we have already noted is a forbidden $k$-subfamily. If $C_3$ corresponds to $k$-subfamily above, then can compute the Laplacian eigenvalues to find that the top two eigenvalues of $C_3 \sqcup K_k$ are both $4$, while the sub-family has four facets. Therefore, it remains a forbidden $k$-subfamily.

Suppose $\ell = 4$ and denote the subfamily corresponding to the cycle by $C_4$. Caboara et al. note that simplicial cycles are either a sequence of facets joined together to form a circle in such a way that all intersections are pairwise disjoint or they are the simplicial cone over such a structure \cite{F2}. For the case where the cycle is of length $4$, these two are equivalent. Therefore, we can construct a $k$-family $C_4$ by first starting with the cycle graph on four nodes and taking the simplicial cone until we have a $k$-family. We have already noted how the Laplacian spectrum changes under the coning operation. Namely, $\lambda_1(C_4) = k+2$ and $\lambda_2(C_4) = k$. On the other hand, $C_4$ has four facets. Therefore, it is a forbidden $k$-subfamily for $k \geq 3$ since the inequality in Lemma \ref{lm:forbidden} requires that:
$$ k + 2 + k > 4(k-1)$$
which does not hold if $k \geq 3$.

Suppose $\ell = 5$. We note that $S_2 \sqcup K_{k}$ is a $k$-subfamily. The two largest eigenvalues of this $k$-subfamily are $k+1$ and $k$. Therefore, it is a forbidden $k$-subfamily for $k > 3$. For the case where $k = 3$, we will instead focus on the $k$-subfamily corresponding to the cycle itself, $C_5$. This subfamily has two largest eigenvalues $4.618$ both. Therefore, it is a forbidden $k$-subfamily since $C_5$ has five facets and $k \geq 3$.

\end{proof}

Combining the lemma for $k$-families with matching numbers $1$ and $2$, we get the desired result:

\begin{thm}[Second Partial Sum]\label{thm:sps}
The conjectural family of inequalities given in Conjecture \ref{red} holds for all $S$ for the second partial sum.
\end{thm}

\section*{Acknowledgments}

The author would like to thank Victor Reiner for his extensive feedback and guidance throughout the course of this work. The author would also like to thank Joshua Pfeffer and Thomas McConville for helpful discussions, especially for the proof for shifted simplicial complexes. The author has also benefited greatly from discussions with Laszlo M. Lovasz, Shravas Rao, and Richard Stanley. Part of this work was completed as part of an REU at the University of Minnesota, where support was provided by the NSF/DMS-1148634 and the RTG grants. This work has also been partially supported by scholarships from Facebook, Google, and the Howard Hughes Medical Institute.

\bibliographystyle{plain}
\bibliography{refs}

\end{document}